\documentclass[a4paper,12pt]{article}
\usepackage{epsfig,amssymb,amsmath,amsthm,fullpage,wrapfig,pstricks}
\newtheorem{theorem}{Theorem}
\newtheorem{corollary}{Corollary}
\newtheorem{lemma}{Lemma}
\newtheorem{proposition}{Proposition}
\newtheorem{remark}{Remark}
\begin{document}

\title{The Hausdorff dimension of a class of\\random self-similar fractal trees\\}
\author{D.A. Croydon\footnote{Dept of Statistics,
University of Warwick, Coventry CV4 7AL, UK;
{d.a.croydon@warwick.ac.uk}.}}
\date{15 May 2007}
\maketitle
\begin{abstract}
In this article a collection of random self-similar fractal dendrites is constructed, and their Hausdorff dimension is calculated. Previous results determining this quantity for random self-similar structures have relied on geometric properties of an underlying metric space or the scaling factors being bounded uniformly away from zero. However, using a percolative argument, and taking advantage of the tree-like structure of the sets considered here, it is shown that conditions such as these are not necessary. The scaling factors of the recursively defined structures in consideration form what is known as a multiplicative cascade, and results about the height of this random object are also obtained.
\end{abstract}

\section{Introduction}

There is now a substantial literature focusing on the geometrical and analytical properties of self-similar fractals, which are commonly described as the unique non-empty compact subset $K\subseteq X$ satisfying $K=\cup_{i\in S} \psi_i(K)$, where $(\psi_i)_{i\in S}$ is a finite collection of contractions on an underlying complete metric space $(X,d)$. The existence and uniqueness of $K$ is guaranteed by an extension of the usual contraction principle for complete metric spaces (\cite{Kigami}, Theorem 1.1.4, for example). A fundamental problem in this area is to calculate the Hausdorff dimension, $\mathrm{dim}_HK$ say, of the self-similar fractal $K$, and in a wide class of examples it is now known (\cite{Kigami}, Corollary 1.5.9) that $\mathrm{dim}_HK$ is the unique positive $\alpha$ solving $\sum_{i\in S}r_i^\alpha=1$, where $(r_i)_{i\in S}$ are the contraction ratios of $(\psi_i)_{i\in S}$. Additionally, various stochastic versions of this result have been investigated. For example, when the underlying metric space $(X,d)$ is finite dimensional Euclidean space, in \cite{Mauldin} a random self-similar set $K$ satisfying $K=\cup_{i\in S}w(i)K_i$, where $(w(i))_{i\in S}$ is a random (finite or countable) collection of scaling factors and $(K_i)_{i\in S}$ are independent copies (up to translation) of $K$, independent of $(w(i))_{i\in S}$, is constructed and (assuming $K$ is non-empty and $(w(i)K_i)_{i\in S}$ fulfils a non-overlapping condition) the Hausdorff dimension of $K$ is shown to be $P$-a.s. equal to the unique positive $\alpha$ solving
\begin{equation}\label{alpha}
E\sum_{i\in S}w(i)^\alpha=1,
\end{equation}
which obviously reduces to the deterministic equation when $w(i)=r_i$, $P$-a.s. Note that we are assuming that all the random variables are defined on an underlying probability space with probability measure $P$, and $E$ is the expectation under $P$. A similar result was proved independently in \cite{Falconer}.

The aim of this article is to obtain similar Hausdorff dimension results to those discussed above for a class of random self-similar trees that do not fit into the usual framework for random self-similar sets. First though, it is necessary to deal with questions concerning their construction. Our starting point, rather than to define a random set directly, is to consider a fixed fractal subset of $\mathbb{R}^2$ and build a random metric upon it. More precisely, for $(x,y)\in\mathbb{R}^2$, set
\[\psi_1(x,y):=\frac{1}{2}(1-x,y),\hspace{10pt}\psi_2(x,y):=\frac{1}{2}(1+x,-y),\hspace{10pt}\psi_3(x,y):=\left(\frac{1}{2}+cy,cx\right),\]
for some constant $c\in(0,1/2)$, and define $T$ to be the unique
non-empty compact set satisfying $T=\bigcup_{i\in S}\psi_i(T)$, where we now define $S:=\{1,2,3\}$. The set $T$, shown in Figure \ref{ssd}, is easily checked to be a {\it dendrite}, by which we mean that it is an arc-wise connected topological space containing no subset homeomorphic to the circle. Although the Euclidean metric is important for its construction, we are only interested in $T$ as a topological space. Indeed, the Hausdorff dimension of $T$ with respect to the intrinsic random metric we construct upon it can be strictly larger than 2.

\begin{figure}[t]
\centering
\scalebox{0.7}{
\includegraphics{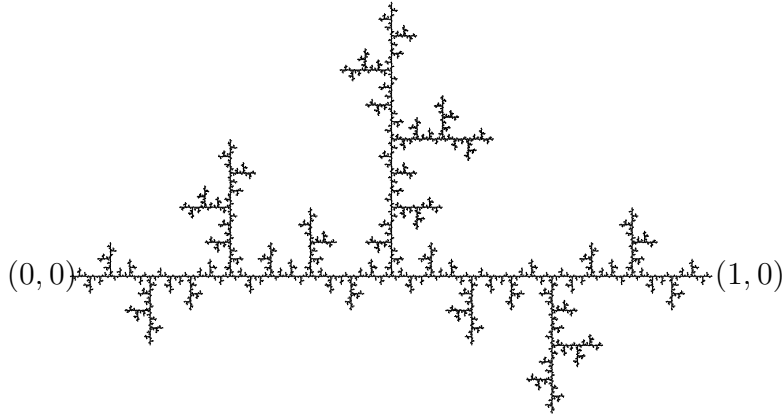}}
\put(-265,50){$(0,0)$}
\put(0,50){$(1,0)$}
\caption{Self-similar dendrite $T$.}
\label{ssd}
\end{figure}

Heuristically, we build a random metric on $T$ by first supposing that the distance between the ``edge'' from $(0,0)$ to $(1,0)$ is of length 1. We then replace this by three randomly scaled copies of the edge with new lengths given by $(w(i))_{i\in S}$, see Figure \ref{edgerep}, and continue inductively to replace edges independently of each other by triples of scaled edges, with the relevant scaling factors having the same distribution as $(w(i))_{i\in S}$. As the number of inductive steps increases our discrete approximations eventually fill out a dense subset of $T$, and (under certain distributional conditions on the scaling factors) calculating the ``limiting distance'' between points yields a metric $R$ on $T$ such that $(T,R)$ is a compact metric space, and the topology induced by $R$ on $T$ is the same as the original (Euclidean) one, $P$-a.s. See Section \ref{ssdsec} for full details.

\begin{figure}[t]
\begin{center}
\begin{pspicture}(10,2.75)
\psline(.5,.25)(3.5,0.25)
\rput(0,.25){$(0,0)$}
\rput(4,.25){$(1,0)$}
\rput(2,0){$1$}
\psline[arrows=->](4.5,.75)(5.5,.75)
\psline(6.5,.25)(9.5,.25)
\psline(8.0,.25)(8,1.75)
\rput(6,0.25){$(0,0)$}
\rput(10,0.25){$(1,0)$}
\rput(8,2){$(\frac{1}{2},\frac{1}{2})$}
\rput(7.25,0){$w(1)$}
\rput(8.75,0){$w(2)$}
\rput(8.5,1){$w(3)$}
\end{pspicture}
\end{center}
\caption{Edge replacement procedure.}
\label{edgerep}
\end{figure}
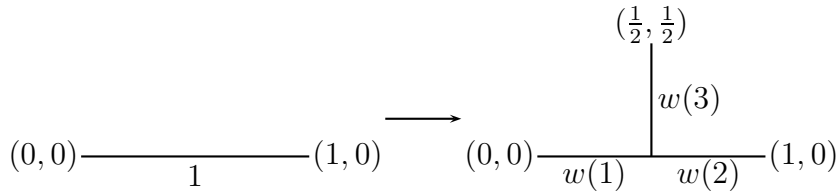

Importantly, we do not assume that $w(1)+w(2)\equiv1$, $P$-a.s., and, as a result of this, the distance in $(T,R)$ between $(0,0)$ and $(1,0)$ depends, in general, on all of the steps in the inductive procedure. For reasons related to the construction of the metric $R$ as a so-called resistance metric, we call this limiting distance between $(0,0)$ and $(1,0)$, $R_\emptyset$ say, a resistance perturbation. Similarly, the distance between $(0,0)$ and $(\frac{1}{2},0)$ in $(T,R)$ is not simply $w(1)$, but equal to $w(1)R_1$, where $R_1$ is the resistance perturbation associated with this ``edge'' of $T$ and has the same distribution as $R_\emptyset$. The fact that $T$ is a dendrite means that we can characterise these resistance perturbations in a convenient way that allows us to deduce their distributional properties and perform calculations with them.

In Section \ref{dimension}, we provide conditions on $(w(i))_{i\in S}$ that allow us to deduce the Hausdorff dimension of $(T,R)$. Proving that the $\alpha$ defined by (\ref{alpha}) is an upper bound for the Hausdorff dimension it straightforward, and requires no further conditions than those used for the construction of $(T,R)$. On the other hand, when obtaining a lower bound for the dimension of a fractal defined in a recursive fashion, it is often a problem when parts of the fractal become small too quickly, and this is the case here. For the proof in \cite{Mauldin} of the result described above, the underlying Euclidean geometry of the random sets being considered is critical, but since our fractal trees are not embedded into any fixed metric space there is no easy translation of this argument to our setting. Another common assumption for proving a Hausdorff dimension lower bound and related results is that the scaling factors are bounded uniformly away from zero \cite{Falconer}, \cite{HamJon}. We show in Theorem \ref{lowerboundbounded} one application of this condition in our setting. More interestingly, though, is that knowledge of the geometry of the trees in consideration allows us to avoid a uniform lower bound; in fact, we shall require only that clusters of small scaling factors are not too large. We construct a random graph approximation to $T$ and use a percolation argument upon this to show that this is the case when the scaling factors are independent and their distributions satisfy a simple polynomial tail bound at 0, see Theorem \ref{hausdorfflower2}.

Let us complete the discussion of our results related to random self-similar dendrites by outlining a pair of examples. Firstly, one choice of random scaling factors that does not fit into any previously studied set-up but satisfies conditions that allow us to calculate the Hausdorff dimension of the fractal $(T,R)$ is if we suppose $(w(i))_{i\in S}$ are independent identically-distributed $U(0,1)$ random variables. It is easy to check that the Hausdorff dimension is 2 in this case. Note in particular that $w(1)+w(2)\neq 1$, $P$-a.s.; consequently there are indeed non-trivial resistance perturbations in this case. Secondly, although we can not calculate the Hausdorff dimension using the techniques of this article, when $(w(i))_{i\in S}$ are the square-roots of a Dirichlet $(\frac{1}{2},\frac{1}{2},\frac{1}{2})$ triple, the construction of $(T,R)$ is of interest in its own right. The reason for this is explained in detail in \cite{hamcroy}, where it is proved that in this case $(T,R)$ is a version of the continuum random tree of Aldous, which is an important random dendrite with connections to many other stochastic tree-like objects \cite{Aldous2}.

The family of random scaling factors that we use to construct the random metric space $(T,R)$ form what is known as a multiplicative cascade, which is a probabilistic structure that has been studied extensively, see \cite{Falconer2}, \cite{Liu}, \cite{LiuRou}, \cite{Mandelbrot}. Much of this previous work has concentrated on investigating properties of an associated tree-martingale limit (see Section \ref{multcasc} for a definition), and we add to this body of knowledge by proving a tail bound at 0 for this random variable (Proposition \ref{lefttailprop}). We also define the height of a (generalised) multiplicative cascade, and derive a simple condition that yields the finiteness of this quantity and its moments, see Corollary \ref{heightupperbound} and Theorem \ref{heighttheorem}.

Finally, let us remark that, although only one particular self-similar fractal dendrite is considered in this article, at the expense of some additional notation relatively minor refinements of the proofs used here allow the results we obtain to be extended to any post-critically finite self-similar fractal dendrite (see \cite{Kigami} for a definition of a post-critically finite self-similar fractal). A treatment of the general case appears in \cite{thesis}.

\section{Multiplicative cascade results}\label{multcasc}

We first introduce an address space to label various objects in the discussion. Fix a finite index set $S$, let $N:=|S|$, and define the ``shift-space'' of infinite sequences ${\Sigma}:=S^\mathbb{N}$. The corresponding finite sequences are denoted by, for $n\geq 0$,
\[\Sigma_n:=S^n,\hspace{20pt}{\Sigma}_*:=\bigcup_{m=0}^{\infty}\Sigma_m,\]
where $S^0:=\{\emptyset\}$. For $i\in\Sigma_m$, $j\in\Sigma_n$, $k\in \Sigma$, we write $ij=$ $(i_1,\dots,i_m,j_1,\dots, j_n)$ and $ik=$ $(i_1,\dots,i_m,k_1,k_2\,\dots)$. For $i\in\Sigma_*$, we denote by $|i|$ the integer $n$ such that $i\in\Sigma_n$ and call this the length of $i$. For $i\in\Sigma_n\cup\Sigma$, $n\geq m$, the truncation of $i$ to length $m$ is written as $i|m=(i_1,\dots, i_m)$.

We define a {\it multiplicative cascade} to be a family of random variables $(w(i))_{i\in\Sigma_*\backslash\{\emptyset\}}$ which take values in $[0,1]$ such that, for $i\in \Sigma_*\backslash\{\emptyset\}$, the $N$-tuples $(w(ij))_{j\in S}$ are independent copies of $(w(j))_{j\in S}$. The multiplicative cascade has a naturally associated filtration $(\mathcal{F}_n)_{n \geq 0}$, defined by $\mathcal{F}_n:=\sigma(w(i):|i|\leq n)$. Throughout our arguments, we use the function $F(\theta):=E\left( \sum_{i\in S}w(i)^\theta \right)$, for $\theta >0$, which is decreasing, continuous and satisfies
\begin{equation}\label{Flimit}
F(\theta) \rightarrow \sum_{i\in S} P(w(i)=1)\mbox{\hspace{20pt} as }\theta \rightarrow\infty.
\end{equation}

Furthermore, we introduce the notation $l(i):=w(i|1)w(i|2)\dots w(i)$ for $i\in \Sigma_* \backslash \{\emptyset\}$, and $l(\emptyset)=1$; and use this to define the so-called {\it tree-martingale} (this term was coined in \cite{Falconer2}), $(M^\theta(n))_{n\geq 0}$, by $M^\theta(n):= \sum_{i\in \Sigma_n} l(i)^\theta F(\theta)^{-n}$. It is straightforward to check that, for each $\theta > 0$, $(M^\theta(n))_{n \geq 0 }$ is an $(\mathcal{F}_n)_{n \geq 0}$ martingale. In particular, $E\left(M^\theta(n)\right)=E\left(M^\theta(0)\right)=1.$ By the almost sure martingale convergence theorem, this implies $M^\theta(n)\rightarrow M^\theta\mbox{ as }n\rightarrow\infty,$ $P\mbox{-a.s.},$ for some random variable, $M^\theta$, with $EM^\theta\in [0,1]$. By relabeling, we have the same distributional properties for the random variables, $(M^\theta_i)_{i\in\Sigma_*}$, defined by
\begin{equation}\label{mitheta}
M^\theta_i = \lim_{n\rightarrow \infty}\frac{\sum_{j\in\Sigma_n}l(i,j)^\theta}{l(i)^\theta F(\theta)^n}.
\end{equation}
It is not difficult to check that for each $n$, $(M_i^\theta)_{i\in\Sigma_n}$ is a collection of independent, identically distributed random variables, independent of $\mathcal{F}_n$. We also have that the following identity holds:
\begin{equation}\label{mthetaident}
M^\theta= \sum_{i\in\Sigma_n} \frac{l(i)^\theta}{F(\theta)^n}M^\theta_i.
\end{equation}

To generalise our cascade model, we introduce random perturbations of the $l(i)$, denoted by $(X_i)_{i\in\Sigma_*}$. We shall assume that the $X_i$ are identically distributed non-negative random variables satisfying: $E(X_i^\theta)<\infty$ for every $\theta>0$; and also $X_i\bot \mathcal{F}_{|i|}$ for all $i\in\Sigma_{*}$,
where $\bot$ is taken to mean ``is independent of''. The reason for the introduction of the factors $(X_i)_{i\in \Sigma_*}$ will become apparent later, since perturbations with these properties arise naturally in the construction of our self-similar dendrite.

We can consider the cascade model as a weighted graph tree, rooted at $\emptyset$ with vertex set $\Sigma_*$ and edge set $\{\{i, i|(|i|-1)\}:i\in \Sigma_*\backslash\{\emptyset\}\}$; where the edge $\{i, i|(|i|-1)\}$ has weight $l(i)X_i$. For two vertices in $\Sigma_*$, we define the distance between them to be the sum of edge distances along the shortest path in the graph. We then define the {\it height} of the tree to be
\[H=\sup_{i\in\Sigma} \sum_{n=0}^\infty l(i|n)X_{i|n}.\]
The usual definition of tree height (the supremum of distances of vertices from the root) actually has the sum index starting from 1, but we shall find this slightly adjusted definition more useful in later sections.

Our first main result about multiplicative cascades is Theorem \ref{mainresult}, a simple corollary of which gives a sufficient condition for the expected height $EH$ to be finite. In Theorem \ref{heighttheorem}, we deal with the unperturbed case and show that the condition is necessary in this case. We start by estimating how fast the edge lengths $l(i)X_i$ decay as $|i|\rightarrow \infty$.

\begin{lemma} \label{heightlemma} Suppose $\sum_{i\in S} P(w(i)=1)<1$ and fix $d\geq 0$, then\\
(a) there exist constants $c<\infty$, $\alpha_1 \in (0,1)$ such that
\[E\left(\left(\sup_{i\in\Sigma_n}l(i)X_i\right)^d\right)\leq c\alpha_1^n\mbox{,\hspace{20pt}}\forall n\in \mathbb{N}. \]
\\
(b) there exists a positive, finite random variable $A$ and (deterministic) $\alpha_2 \in (0,1)$ such that
\[\sup_{i\in\Sigma_n}l(i)X_i \leq A\alpha_2^n\mbox{,\hspace{20pt}}\forall n\in \mathbb{N}\mbox{,\hspace{20pt}}P\mbox{-a.s.}\]
\end{lemma}
\begin{proof} To prove (a) we first look for bounds on the tail of the distribution of $\sup_{i\in\Sigma_n}l(i)$. Applying Markov's inequality, the definition of $M^\theta(n)$ and the independence assumption of the $X_i$s we obtain, for $\theta >0$,
\begin{equation}\label{Psupliub}
P\left(\sup_{i\in\Sigma_n}l(i)X_i\geq \lambda\right)\leq P\left(\sum_{i\in\Sigma_n}l(i)^\theta X_i^\theta \geq \lambda^\theta \right)  \leq    \lambda^{-\theta}E\left(X_\emptyset^\theta\right) F(\theta)^n
\end{equation}
The condition $\sum_{i\in S} P(w(i)=1)<1$ and (\ref{Flimit}) imply that we can find $\theta_0>d$ large enough so that $F(\theta_0)<1$. Let $x:=\|X_\emptyset\|_{\theta_0}$ ($<\infty$ by assumption) and define $\lambda_n:=xF(\theta_0)^{n / \theta_0}$, which is less than 1 for $n\geq n_0$ for some $n_0\in\mathbb{N}$. Assume for now that $n \geq n_0$. For $\lambda\geq\lambda_n$, the upper bound at (\ref{Psupliub}) is $\leq 1$ and so is non-trivial and, for $\lambda<\lambda_n$, we merely use the fact that we are trying to bound a probability to deduce
\begin{eqnarray*}
E\left( \left(\sup_{i\in\Sigma_n}l(i)X_i\right)^d \right) & = & \int_0^\infty d\lambda^{d-1} P\left(\sup_{i\in\Sigma_n}l(i)X_i\geq \lambda\right)d \lambda \\
&\leq&\int_0^{\lambda_n}d\lambda^{d-1}d\lambda+ \int_{\lambda_n}^\infty d\lambda^{d-1-\theta_0}x^{\theta_0}F(\theta_0)^nd\lambda \\
&=&\frac{x^d\theta_0}{\theta_0-d}F(\theta_0)^{\frac{dn}{\theta_0}}
\end{eqnarray*}
Hence taking $\alpha_1=F(\theta_0)^{\frac{d}{\theta_0}}$ and $c$ suitably large gives us part (a) of the lemma.

To prove (b) we again look to bound the tail probability of $\sup_{i\in\Sigma_n}l(i)X_i$. We proceed as above to obtain the bound
$P(\sup_{i\in\Sigma_n}l(i)X_i\geq \lambda^n)\leq\lambda^{-n\theta}E\left(X_\emptyset^\theta\right) F(\theta)^n$. If we fix $\theta=\theta_0$ we can find a $\lambda_0\in(0,1)$ such that $\lambda_0^{-\theta_0}F(\theta_0)<1$ and so
\[\sum_{n=0}^{\infty} P\left(\sup_{i\in\Sigma_n}l(i)X_i\geq \lambda_0^n\right)\leq E\left(X_\emptyset^\theta\right) \sum_{n=0}^\infty \left(\lambda_0^{-\theta_0}F(\theta_0)\right)^n <\infty.\]
An application of the Borel-Cantelli lemma then gives us part (b) of the lemma.
\end{proof}

In proving Theorem \ref{mainresult} we will apply the following elementary lemma, which we state without proof.

\begin{lemma} \label{holder}
Let $(x_n)_{n\geq0}$ be a sequence of non-negative real numbers, then
\[\left(\sum_{n=0}^\infty x_n\right)^d\leq c\sum_{n=0}^\infty x_n^d (1+n)^d,\hspace{20pt}\forall d\geq 1,\]
where $c$ is a constant depending only on $d$.
\end{lemma}

\begin{theorem} \label{mainresult} Let $\sum_{i\in S}P(w(i)=1)<1$ and $d\geq0$, then
\[E\left(\left(\sum_{n=0}^\infty \sup_{i\in\Sigma_n} l(i)X_i \right)^d\right)<\infty.\]
\end{theorem}
\begin{proof} For $d\geq1$, applying Lemmas \ref{heightlemma} and \ref{holder} yields
\[E\left(\left(\sum_{n=0}^\infty \sup_{i\in\Sigma_n} l(i)X_i \right)^d\right)\leq c_1 E\left(\sum_{n=0}^\infty \left(\sup_{i\in\Sigma_n} l(i)X_i \right)^d(1+n)^d\right)
\leq c_2\sum_{n=0}^\infty (1+n)^d\alpha_1^n,\]
which is clearly finite. For $d\in[0,1)$, the result follows from the case $d=1$ by applying the inequality $x^d\leq 1+x$, $\forall x\geq 0$.
\end{proof}

\begin{corollary} \label{heightupperbound} Let $\sum_{i\in S}P(w(i)=1)<1$ and $d\geq0$, then $E(H^d)<\infty$.
\end{corollary}
\begin{proof} After noting that $\sup_{i\in\Sigma}\sum_n l(i|n)X_{i|n}$ $\leq\sum_n \sup_{i\in\Sigma_n} l(i)X_i$, this result follows from Theorem \ref{mainresult}.
\end{proof}

We now show how for unperturbed cascades the condition $\sum_{i\in S} P(w(i)=1)<1$ is also necessary for finite moments of $H$.

\begin{theorem} \label{heighttheorem} Assume $X_i\equiv1$, $\forall i \in \Sigma_*$, then
\\
(a) if $\sum_{i\in S}P(w(i)=1)<1$, $EH^d<\infty$, $\forall d\in\mathbb{R}$.
\\
(b) if $\sum_{i\in S}P(w(i)=1)=1$, $EH=\infty$.
\\
(c) if $\sum_{i\in S}P(w(i)=1)>1$, $P(H=\infty)>0$.
\end{theorem}
\begin{proof} Assume $\sum_{i\in S} P(w(i)=1)<1$. Clearly, if $X_i\equiv1$, $\forall i \in \Sigma_*$, then $(X_i)_{i\in\Sigma_*}$ satisfy the conditions that enable us to  apply the previous results of the section. Hence, if $d\geq0$, $E(H^d)<\infty$ follows from Corollary \ref{heightupperbound}. We note that, because $l(\emptyset)=1$, $H\geq1$. Hence, for $d<0$, $E(H^d)\leq1<\infty$. Thus part (a) holds.

We now construct a Galton-Watson branching process related to our tree. Given $(w(i))_{i\in \Sigma_*\backslash{\{\emptyset\}}}$, we define $\tilde{w}(i) := 1_{\left\{w(i)=1\right\}}$ and $\tilde{l}(i) := \prod_{n=1}^{|i|}\tilde{w}(i|n)$. It is then easy to check that if $Z_n:= \sum_{i\in\Sigma_n} \tilde{l}(i)$, then $(Z_n)_{n\geq0}$ is a Galton-Watson process. Importantly, if $Z_n>0$ then we must have $\tilde{l}(i)=1$ for some $i\in\Sigma_n$, and so $\tilde{l}(i|m)=1$ for $1\leq m\leq n$. Consequently $H\geq n+1$. This means that $\left\{Z_n>0\right\}\subseteq\left\{H\geq n+1\right\}$ and so we can use known results about the extinction of the Galton-Watson process to infer results about $H$. In particular, it may be calculated that $EZ_1=\sum_{i\in S} P(w(i)=1)$. To prove part (c) we note that if $\sum_{i\in S} P(w(i)=1)>1$, then $EZ_1>1$ and so the branching process is supercritical and survives with positive probability (for a proof of this see, for example, \cite{A-N}). Thus we have $Z_n>0$, $\forall n \geq 0$, with strictly positive probability and this implies that $P(H=\infty)=P(H \geq n$ , $\forall n\geq0)>0$.

Assume now $EZ_1=\sum_{i\in S} P(w(i)=1)=1$. It is a standard result that the extinction time of a critical Galton-Watson process with finite offspring variance has an infinite expectation. If $X$ is the extinction time of our Galton-Watson process, then we have $\left\{X>n\right\}=\left\{Z_n>0\right\}\subseteq\left\{H\geq n+1\right\}$ and so $H>X$. Since $EX=\infty$, it follows that (b) holds.
\end{proof}

\begin{remark}
At criticality, the Galton-Watson process exhibits two types of behaviour. First, there is the trivial case when $Z_1=1$, $P$-a.s. This implies that $Z_n=1$, $\forall n$, and so the process survives. This is mirrored in the multiplicative cascade when we have $\sum_{i\in S} P(w(i)=1)=1$ and $P(\sup_{i\in S} w(i)=1)=1$. In this case $H=\infty$, $P$-a.s. In the non-trivial case, $P(Z_1=1)<1$, the Galton-Watson process becomes extinct with probability 1. It follows that we can also find a multiplicative cascade with $\sum_{i\in S} P(w(i)=1)=1$ and $H<\infty$, $P$-a.s. The problem of whether we have $H<\infty$ in the general non-trivial case is left open.
\end{remark}
\begin{remark}
If $N=\infty$, there exist random variables with $\sum_{i\in S}P(w(i)=1)<1$ and also $\sup_{i\in S} w(i) =1$, $P$-a.s. In this case we have $H=\infty$, $P$-a.s., and so the theorem does not hold in general when $N=\infty$.
\end{remark}
\begin{remark}
The condition $w(i)\in [0,1]$, $P$-a.s., is not essential. For example, if we let $(w(i))_{i\in S}$ be independent, identically distributed $U[0,x]$ random variables and $N\geq 2$, then by mimicking the proofs of Lemma \ref{heightlemma} and Theorem \ref{heighttheorem} it can be shown that $EH^d<\infty$ for some values of $x$ strictly greater than 1 (the largest $x$ for which this is true is $1+O(N^{-1})$).
\end{remark}

The right-tail of the distribution of the tree-martingale limit $M^\theta$ has been considered by various authors, including Liu, who has proved a widely applicable result demonstrating exponential tails \cite{Liu}. However, we will later require that $M^\theta$ also has finite negative moments of some order for certain values of $\theta$, and so we will also need information about its tail at 0. The aim of the remainder of this section is to provide some simple conditions on the distribution of $(w(i))_{i\in S}$ that allow us to deduce suitable estimates on the distribution of $M^\theta$.

Let us start by assuming that $N>1$, $\sum_{i\in S}P(w(i)=1)<1$, and $P(w(i)=0)=0$ for every i. It follows that $F(\theta)$ is a strictly decreasing continuous function with $F(0)=N>1$ and $F(\theta)\rightarrow\sum_{i\in S} P(w(i)=1)<1$ as $\theta \rightarrow\infty$. Hence there exists a unique, strictly positive solution to the equation $F(\theta)=1$. We shall denote this solution $\alpha$. In the next lemma we present a few simple properties of $M^\theta$ when  $\theta\leq\alpha$ that we will apply later.

\begin{lemma} \label{mthetal2} Let $N>1$, $\sum_{i\in S}P(w(i)=1)<1$, and $P(w(i)=0)=0$ for every i, and $\alpha$ be the unique solution to $F(\theta)=1$. If $\theta\leq \alpha$, then
\\
(a) $P(M^\theta \in (0,\infty))=1$.
\\
(b) $E((M^\theta)^d) < \infty$, $\forall d\geq0$.
\end{lemma}
\begin{proof} For $\theta\leq \alpha$, we can apply results of \cite{Liu} to deduce that $M^\theta(n)\rightarrow M^\theta$ in mean, and so $EM^\theta=1$. Furthermore, the decomposition of $M^\theta$ given at (\ref{mthetaident}) allows us to deduce that $P(M^\theta=0)^N=P(M^\theta=0)$, thus $P(M^\theta=0)\in\{0,1\}$. Combining these two facts immediately implies parts (a). Part (b) is also proved in \cite{Liu}.
\end{proof}

Our main result for the tail at 0 of $M^\theta$ is proved using the methods of Barlow and Bass \cite{BarBas}. Their arguments allow a polynomial estimate for the distribution function of the limit random variable to be improved to an exponential bound. We shall prove our main result only in the case $\theta=\alpha$, as this eases notation by removing the factors $F(\theta)$ (recall $F(\alpha)= 1$). However, the techniques used will actually apply whenever $\theta\leq\alpha$. The first step is to deduce the necessary polynomial bound.

\begin{lemma} If $(w(i))_{i\in S}$ satisfy the conditions of Lemma \ref{mthetal2}, and fix $\beta>0$, $\varepsilon\in(0,1)$, then there exists a constant $c$ such that $P(M^\alpha\leq x)\leq \varepsilon+cx^\beta$ for every $x\geq 0$.
\end{lemma}
\begin{proof} Fix $\beta>0$, $\varepsilon\in(0,1)$. By Lemma \ref{mthetal2}, the tree-martingale limit $M^\alpha$ is non-zero, $P$-a.s. Hence $P(M^\alpha\leq x)\rightarrow 0$ as $x\rightarrow 0$. In particular, there is an $x_0>0$ such that $P(M^\alpha\leq x_0)\leq \varepsilon$. Thus, for $x\geq 0$, $P(M^\alpha\leq x)\leq \varepsilon+x_0^{-\beta}x^\beta$.
\end{proof}

We can now prove the exponential tail bound at 0 for $M^\alpha$ under the additional assumptions of independence and finiteness of negative moments of the $(w(i))_{i\in S}$.

\begin{proposition} \label{lefttailprop} If $(w(i))_{i\in S}$ satisfy the conditions of Lemma \ref{mthetal2}, are independent, and also $\max_{i\in S}E(w(i)^{-\beta})<\infty$ for some $\beta>0$, then for some constants $c, \gamma \in (0, \infty)$,
\[P(M^\alpha \leq x)\leq e^{-c x^{-\gamma}},\hspace{20pt}\forall x\geq 0.\]
\end{proposition}
\begin{proof} Let $\beta>0$ be a constant for which $\sup_{i\in S}E(w(i)^{-\beta})<\infty$ and fix $\varepsilon\in(0,1)$. By the previous lemma, we can find a constant $c_1$ such that $P(M^\alpha\leq x)\leq \varepsilon +c_1 x^\beta$, for all $x\geq 0$. Applying the relevant independence assumptions and the fact that $M_i^\alpha\sim M^\alpha$, we can deduce from this that, for all $x\geq 0$, $i\in\Sigma_n$,
\[P(l(i)^\alpha M_i^\alpha \leq x)=E\left(P\left(l(i)^\alpha M_i^\alpha \leq x\:\vline\:\mathcal{F}_n\right)\right)\leq  E\left( \varepsilon+c_1\frac{x^{\beta/\alpha}}{l(i)^\beta}\right)
\leq \varepsilon+c_1 c_2^nx^{\beta/\alpha},\]
where $c_2:=\max_{i\in S}E(w(i)^{-\beta})\vee (N+1)$. By writing $M^\alpha=\sum_{i\in\Sigma_n}l(i)^\alpha M_i^\alpha$, one can check that the conditions of \cite{BarBas}, Lemma 1.1 hold, which implies that
\begin{equation}\label{expupperbound}
P(M^\alpha\leq x) \leq \exp\left( c_3 (c_2 N)^{n/2}x^{\beta/2\alpha} +N^n\ln \varepsilon\right),\hspace{40pt}\forall x\geq 0,
\end{equation}
for some constant $c_3$. We now look to choose $n$ in a way that will give us the control we require over this bound. Define $n_0=n_0(x)$ to be the unique solution to $(c_2/N)^{n_0/2}=-\ln\varepsilon/x^{\beta/2\alpha}c_3$, and then set $n=\lfloor n_0-1 \rfloor$. We have $c_2>N$ and so we can find $c_4>0$ such that $Nc_2^{-1}\leq (1-c_4)^2$. Consequently, because $n-n_0 \in (-2,-1]$, we have
\[(c_2N)^{(n-n_0)/2}-N^{n-n_0}=N^{n-n_0}\left(\left(c_2N^{-1}\right)^{(n-n_0)/2}-1\right)\leq -c_4 N^{-2}.\]
By the choice of $n_0$, our upper bound, (\ref{expupperbound}), now becomes
\begin{eqnarray*}
\ln P(M^\alpha\leq x)&\leq&c_3(c_2N)^{n_0/2}x^{\beta/2\alpha}(c_2N)^{(n-n_0)/2}+N^{n_0}N^{n-n_0}\ln \varepsilon\\
&=&-N^{n_0} \left((c_2 N)^{(n-n_0)/2}-N^{n-n_0}\right)\ln \varepsilon\\
&\leq&c_4 N^{n_0-2}\ln \varepsilon\\
&=&c_4 N^{-2} \ln \varepsilon \left(\frac{-\ln \varepsilon}{x^{\beta/2\alpha}c_3}\right)^\frac{2\ln N}{\ln c_2 - \ln N},
\end{eqnarray*}
and the result follows.
\end{proof}

The final result of this section is a simple corollary of Proposition \ref{lefttailprop}.

\begin{corollary} \label{rdnegmoments} Under the assumptions of Proposition \ref{lefttailprop}, $E\left((M^\alpha)^{-d}\right)<\infty$ for every $d\geq0$.
\end{corollary}

\section{Random self-similar fractal dendrites}\label{ssdsec}

In this and the next section we fix $S:=\{1,2,3\}$ and $T$ to be the deterministic dendrite described in the introduction. For a wide class of
self-similar fractals, including $T$, there is now a
well-established approximation procedure for defining an intrinsic quadratic form and associated resistance metric on the relevant space, which we briefly outline in our specific case for the purpose of introducing notation. See \cite{Barlow} and \cite{Kigami} for more details.

First, define the boundary of $T$ to be the two point set $V^0:=\{(0,0), (1,0)\}$, and define an initial Dirichlet form by $D(f,f):=(f(0,0)-f(1,0))^2$, for $f\in C(V^0)$, where, for a countable set, $A$, we denote $C(A):=\{f:A\rightarrow
\mathbb{R}\}$. Furthermore, introduce an increasing family of subsets of $T$
by setting $V^n:=\bigcup_{i\in\Sigma_n}\psi_i(V^0)$, where for
$i\in\Sigma_n$, $\psi_i:=\psi_{i_1}\circ\dots\circ \psi_{i_{n}}$. Given the quadratic form $D$ and a set of scaling factors $(r_i)_{i\in S}$ with $r_i>0$ for each $i\in S$, we can define a Dirichlet form on each of the $V^n$ by setting, for $n\geq 1$,
\begin{equation}\label{en}
\mathcal{E}^n(f,f):=\sum_{i\in\Sigma_n}\frac{1}{r_i}D(f\circ
\psi_i,f\circ \psi_i),\hspace{20pt}\forall f\in C(V^n),
\end{equation}
where $r_i:={r_{i_1}}\dots r_{i_n}$ for $i\in \Sigma_n$. To establish the existence of a non-trivial limit as $n\rightarrow \infty$, we need to place some restrictions on the choice of $(r_i)_{i\in S}$ so that the family $\{(V^n,\mathcal{E}^n)\}_{n\geq 0}$
is compatible in the sense that the trace of $\mathcal{E}^{n+1}$ on $V^n$  is precisely $\mathcal{E}^n$ (\cite{Kigami}, Definition 2.2.1). Some elementary algebra yields that in our case $\{(V^n,\mathcal{E}^n)\}_{n\geq 0}$ is compatible if and only if we assume that $r_1+r_2=1$, and in this case we can take a limit in a sensible way. Specifically, assume that $r_1+r_2=1$ and let
\begin{equation}\label{elim}
\mathcal{E}^*(f,f):=\lim_{n\rightarrow\infty}\mathcal{E}^n(f,f),\hspace{20pt}\forall
f\in\mathcal{F}^*,
\end{equation}
where $\mathcal{F}^*$ is the set of functions on the countable set $V^*:=\bigcup_{n\geq 0}V^n$ for which this limit exists finitely. Note that we have abused notation slightly by using the convention that if a form $\mathcal{E}$ is defined for functions on a set $A$ and $f$ is a function defined on $B\supseteq A$, then we write $\mathcal{E}(f,f)$ to mean $\mathcal{E}(f|_A, f|_A)$. The resulting quadratic form $(\mathcal{E}^*, \mathcal{F}^*)$ is actually a
resistance form (see \cite{Kigami}, Definition 2.3.1), and we can use
it to define a (resistance) metric $R^*$ on $V^*$ by
\begin{equation}\label{resdef}
R^*(x,y)^{-1} = \inf\{\mathcal{E}^*(f,f):f\in\mathcal{F}^*,
f(x)=0,f(y)=1\},
\end{equation}
for $x,y\in V^*$, $x\neq y$, and setting $R^*(x,x)=0$. Finally, if we also assume that $r_i<1$ for each $i\in S$, then $R^*$ extends uniquely to a metric $R$ on $T$ such that $(T,R)$ is the completion of $(V^*,R^*)$, and moreover the topology induced by $R$ on $T$ is the same as that induced by the original (Euclidean) metric.

We will now explain how to randomise the above construction. The scaling factors that we will use to define a sequence of Dirichlet forms on the subsets $(V^n)_{n\geq 0}$ of $T$ will form a multiplicative cascade and, identifying notation with the previous section, be denoted by $(w(i))_{i\in\Sigma_*\backslash\{\emptyset\}}$. In addition to the independence assumptions that we introduced in Section \ref{multcasc} for a multiplicative cascade, we will further suppose that the random variables $w(i)$ are non-zero, $P$-a.s. The following assumptions will also be useful, and, for clarity, we will explicitly state when we apply these.
\smallskip
\\
{\bf (A) } $E(w(1)+w(2))=1$.
\smallskip
\\
{\bf (B) } $\sum_{i\in S} P(w(i)=1)<1$.
\smallskip

Although we would like to simply replace the deterministic scaling factors $r_i$ with the random variables $w(i)$ in a formula similar to (\ref{en}), a sequence of forms defined in this way will not be compatible in general and taking limits would not be straightforward. To deal with this problem, we introduce another collection of random variables
\[R_i:=\lim_{n\rightarrow\infty}\sum_{j\in\{1,2\}^n}\frac{l(ij)}{l(i)},\hspace{20pt}i\in\Sigma_*,\]
which we shall term {\it resistance perturbations}. Clearly these are
identically distributed, and, by appealing to the independence
properties of $(w(i))_{i\in\Sigma_*\backslash\{\emptyset\}}$, various
questions regarding the convergence and distribution of the
$(R_i)_{i\in\Sigma_*}$ may be answered by multiplicative
cascade techniques as discussed in the previous section. In particular, under the assumption (A) we have that: the limit defining $R_i$ exists in $(0,\infty)$, $P$-a.s.; $ER_i^d<\infty$ for every $d\geq 0$; and also $R_i=w(i1)R_{i1}+w(i2)R_{i2}$ for every $i\in\Sigma_*$. Note that (A) is actually necessary for the non-triviality of the $R_i$s, because if $E(w(1)+w(2))\neq 1$, then $R_i\in\{0,\infty\}$, $P$-a.s.

Given a multiplicative cascade of scaling factors satisfying (A), we define a random sequence of Dirichlet forms on the vertex sets $(V^n)_{n\geq 1}$ by, for $n\geq 1$,
\[\mathcal{E}^n(f,f):=\sum_{i\in\Sigma_n}\frac{1}{l(i)R_i}D(f\circ
\psi_i,f\circ \psi_i),\hspace{20pt}\forall f\in C(V^n),\]
where, as before, $D(f,f):=(f(0,0)-f(1,0))^2$ for $f\in C(V^0)$. In analysing this quadratic form, it is natural to consider the graph $(V^n,E^n)$ where we define $E^n:=\{\psi_i(V^0):\:i\in\Sigma_n\}$. The reason for this being that if we place a resistance of $l(i)R_i$ along each edge $\psi_i(V^0)$, then $\mathcal{E}^n$ is the energy operator associated with the resulting electrical network. It is elementary to check that $(V^n,E^n)$ is a graph tree for every $n\geq 0$ by induction, noting that $(V^{n+1},E^{n+1})$ is constructed by joining three graph trees with the same structure as $(V^n,E^n)$ at a single vertex. This fact is convenient as it means that to calculate the resistance between two points in $V^n$ we can simply apply the series law and sum the resistances of the edges on the (unique shortest) path between them in $(V^n,E^n)$. We can use this result to show that the sequence $\{(V^n,\mathcal{E}^n)\}_{n\geq 0}$ is a compatible sequence in the sense described above.

\begin{lemma} If (A) holds, then the sequence $\{(V^n,\mathcal{E}^n)\}_{n\geq 0}$ is compatible, $P$-a.s.
\end{lemma}
\begin{proof} To prove the result it will suffice to prove that the resistance between $\psi_i(0,0)$ and $\psi_i(1,0)$ in the electrical network on $(V^{n+1},E^{n+1})$ is equal to $l(i)R_i$ for every $i\in\Sigma_n$. The path between $\psi_i(0,0)$ and $\psi_i(1,0)$ in $(V^{n+1},E^{n+1})$ is given by the edges $\psi_{i1}(V^0)$ and $\psi_{i2}(V^0)$. Thus the resistance between $\psi_i(0,0)$ and $\psi_i(1,0)$ in the level $n+1$ network is $l(i1)R_{i1}+l(i2)R_{i2}$, which is equal to $l(i)R_i$ and so the proof is complete.
\end{proof}

As a consequence of this lemma, if (A) holds, we can proceed as in the deterministic case to define $(\mathcal{E}^*,\mathcal{F}^*)$ on $V^*$ by the limit at (\ref{elim}). By \cite{Kigami}, Theorem 2.2.6, this is a resistance form and we can define a resistance metric $R^*$ on $V^*$ by (\ref{resdef}), $P$-a.s. Note that the compatibility of the sequence of forms used to define $R^*$ implies that
\begin{equation}\label{edgeres}
R^*(\psi_i(0,0),\psi_i(1,0))=l(i)R_i,\hspace{20pt}\forall i\in\Sigma_*.
\end{equation}
To complete the construction of our random metric on $T$ it remains to show that $R^*$ can be extended uniquely to a metric $R$ on $T$ such that $(T,R)$ is the completion of $(V^*,R^*)$, and that the topology induced by $R$ on $T$ is the same as that induced by the Euclidean metric. Before demonstrating that this is the case, we prove some preliminary results about the diameter of sets of the form $(\psi_i({V}^*))_{i\in\Sigma_*}$ in the metric $R^*$. Let us start with a simple chaining lemma.

\begin{lemma} \label{pathlemma} If $x\in {V}^0$, $y\in{V}^n$, then we can find a sequence $x_0,\dots,x_m$, with $x_0=x$, $x_m=y$, $\{x_{l-1},x_l\}\in\cup_{n'= 0}^n{E}^{n'}$ for $l\in\{1,\dots m\}$, and such that, for $n'\leq n$, $\{x_{l-1},x_l\}\in{E}^{n'}$ for at most two of the $l\in\{1,\dots m\}$.
\end{lemma}
\begin{proof} The proof of this is elementary and we present only an outline here. If $y\in V^{n}$, then we can connect $y$ to a vertex of $V^{n-1}$ using not more than two edges in $E^{n}$. Proceeding inductively gives us (the reverse of) a sequence with the desired properties.
\end{proof}

Throughout the remainder of the article, for a subset $A$ of a metric space $(X,d)$, we shall denote the diameter of $A$ by ${\rm diam}_d A:=\sup\{d(x,y):\:x,y\in A\}$.

\begin{lemma} If we assume (A) and (B), then ${E}(({\rm diam}_{R^*}{V}^*)^d)<\infty$, for all $d\geq0$.
\end{lemma}
\begin{proof} Let $x\in{V}^0$ and $y\in{V}^*$. Necessarily, $y\in{V}^n$ for some $n\geq 0$. Thus the description of paths in ${V}^*$ that was proved in Lemma \ref{pathlemma} and the triangle inequality imply
\[R^*(x,y)\leq 2\sum_{m=0}^n \sup_{i\in\Sigma_m}R^*(\psi_i(0,0),\psi_i(1,0))= 2\sum_{m=0}^n \sup_{i\in\Sigma_m} l(i)R_i,\]
where the equality follows from (\ref{edgeres}). Consequently the diameter of $(V^*,R^*)$ is bounded above by $4\sum_{m=0}^\infty \sup_{i\in\Sigma_m} l(i)R_i$. Applying this estimate, the result can be deduced from Theorem \ref{mainresult}.
\end{proof}

We now introduce random variables $(W_i)_{i\in\Sigma_*}$ to represent the normalised diameters of the sets $(\psi_i({V}^*))_{i\in\Sigma_*}$. In particular, set $W_i:=l(i)^{-1}{\rm diam}_{R^*}\psi_i({V}^*)$ whenever the resistance metric $R^*$ is defined. The definition of the $\sigma$-algebras $\mathcal{F}_n:=\sigma(w(i):\:|i|\leq n)$ should be recalled from Section \ref{multcasc}.

\begin{lemma} \label{wprops} If (A) and (B) hold, then $(W_i)_{i\in\Sigma_*}$ are identically distributed and satisfy: $E(W_i^d)<\infty$, for every $i\in\Sigma_*$, $d>0$; and $W_i\bot\mathcal{F}_{|i|}$ for every $i\in\Sigma_*$.
\end{lemma}
\begin{proof} The self-similarity of $T$ and the multiplicative cascade $(w(i))_{i\in\Sigma_*\backslash\{\emptyset\}}$ immediately implies that $W_i$ has the same distribution as ${\rm diam}_{R^*}{V}^*$ for every $i\in\Sigma_*$. From this we obtain the first claim of the lemma and, when combined with the previous lemma, finite moments of the $W_i$. The remaining result is a straightforward application of the independence properties of the multiplicative cascade $(w(i))_{i\in\Sigma_*\backslash\{\emptyset\}}$.
\end{proof}

It is now easy to show that the diameters of the sets $(\psi_i({V}^*))_{i\in\Sigma_*}$ decrease to 0 uniformly as $|i|\rightarrow\infty$.

\begin{lemma} \label{diamas} If (A) and (B) hold, then $\sup_{i\in\Sigma_n} \mbox{\rm{diam}}_{R^*} \psi_i({V}^*) \rightarrow 0$, $P$-a.s.
\end{lemma}
\begin{proof} This result is a simple corollary of the definition of $(W_i)_{i\in\Sigma_*}$, Lemma \ref{wprops} and the corresponding multiplicative cascade result, Lemma \ref{heightlemma}(b).
\end{proof}

This uniform decay of the diameter of the sets $(\psi_i(V^*))_{i\in\Sigma_*}$ allows us to extend the definition of $R^*$ to the whole of $T$. Note that $T$ is the closure of ${{V}^*}$ with respect to the Euclidean metric (see \cite{Kigami}, Theorem 1.1.7), which we will now denote by $\rho$. Hence, for any $x,y\in T$, there exist sequences $(x_n)_{n\in\mathbb{N}}$, $(y_n)_{n\in\mathbb{N}}$ in ${V}^*$ with $\rho(x_n, x)\rightarrow 0$ and $\rho(y_n, y)\rightarrow 0$. Define
\[R(x,y):=\lim_{n\rightarrow \infty}R^*(x_n, y_n).\]
Before proceeding with the proof of Theorem \ref{rwelldef}, in which we show that this is a sensible definition for $R$, we introduce the notation $T_i:=\psi_i(T)$ and, for $x\in T$,
\begin{equation}\label{tndef}
T_n(x):=\bigcup \{T_i:\:i\in\Sigma_n,\:x\in T_i\},\hspace{10pt}\tilde{T}_n(x):=\bigcup \{T_i:\:i\in\Sigma_n,\:T_i\cap T_n(x)\neq\emptyset\}.
\end{equation}
We will also apply the following result of \cite{Kigami}, Proposition 1.3.6.

\begin{lemma} \label{neighbourhood} $(T_n(x))_{n\in\mathbb{N}}$ forms a base of neighbourhoods of $x$ with respect to $\rho$.
\end{lemma}

\begin{theorem} \label{rwelldef} Assume (A) and (B). $R$ is a well-defined metric on $T$, topologically equivalent to the Euclidean metric, ${P}$-a.s.
\end{theorem}
\begin{proof} Under the assumptions of the theorem, the argument that we give holds ${P}$-a.s. Let $x,y\in T$ and suppose there exist, for $m=1,2$, sequences $(x_n^m)$ and $(y_n^m)$ in ${V}^*$ such that $\rho(x_n^m,x)\rightarrow 0$ and $\rho(y_n^m,y)\rightarrow 0$. Fix $\varepsilon>0$. By Lemma \ref{diamas}, we can choose $n_0\geq 0$ such that $\sup_{i\in\Sigma_n}{\rm diam}_{R^*}\psi_i({V}^*)<\varepsilon/2$, for every $n\geq n_0$. Furthermore, we can use the convergence of the sequences and Lemma \ref{neighbourhood} to show that there exists $n_1\geq n_0$ such that $x_n^m\in T_{n_0}(x)$, $y_n^m\in T_{n_0}(y)$, for $m=1,2$ and $n\geq n_1$. Thus
\[|R^*(x_n^1,y_n^1)-R^*(x_{n'}^m,y_{n'}^m)|\leq R^*(x_n^1,x_{n'}^m)+R^*(y_n^1,y_{n'}^m)<\varepsilon,\]
for $m=1,2$ and $n,n'\geq n_1$. Taking $m=1$, this implies that $(R^*(x_n^1,y_n^1))_{n\geq 0}$ is a Cauchy sequence and has a limit. Taking $m=2$, $n'=n$, this implies that the limit is unique and so the function $R$ is well-defined on $T\times T$. It follows immediately from the fact that $R^*$ is a metric on ${V}^*$ that $R$ is positive, symmetric and satisfies the triangle inequality. To prove $R$ is a metric on $T$, it remains to show that $R(x,y)=0$ implies that $x=y$. We shall prove the stronger claim that $R(x_n, y)\rightarrow 0$ implies that $\rho(x_n, y)\rightarrow 0$.

Suppose $(x_n)_{n\geq 0}$ is a sequence in $T$ with $R(x_n, y)\rightarrow 0$ for some $y\in T$. Fix $\varepsilon>0$, and choose $n_0$ such that $2^{1-n_0}{\rm diam}_\rho T<\varepsilon$. For $z\not\in\tilde{T}_{n_0}(y)$, we must have that $z\in T_i$, $y\in T_j$ for some $i,j\in\Sigma_{n_0}$ with $T_i\cap T_j=\emptyset$. For any $z'\in\psi_i({V}^*)$, $y'\in\psi_j({V}^*)$, using the additivity along paths of the metric $R^*$ and the fact that the sets $(T_k)_{k\in\Sigma_{n_0}}$ only intersect at vertices of ${V}^{n_0}$ (\cite{Kigami}, Proposition 1.3.5), it is possible to show that
\[R(z',y')=R^*(z',y')\geq\inf_{k\in\Sigma_{n_0}} R^*(\psi_{k}(0,0),\psi_k(1,0))=:c.\]
It follows that $R(z,y)\geq c$. Since $c>0$ and $R(x_n,y)\rightarrow 0$, there exists an $n_1$ such that $R(x_n,y)< c$, for all $n\geq n_1$. Consequently $x_n\in\tilde{T}_{n_0}(y)$ for $n\geq n_1$. By our choice of $n_0$ and the fact that the contraction ratios of $(\psi_i)_{i\in S}$ (with respect to $\rho$) are no bigger than $\frac{1}{2}$,  this implies that $\rho(x_n,y) < \varepsilon$, for $n\geq n_1$. Hence $\rho(x_n,y)\rightarrow 0$.

To prove the equivalence of the metrics, it remains to show that, for all sequences $(x_n)_{n\in\mathbb{N}}$ in $T$ with $\rho(x_n, x)\rightarrow 0$ for some $x\in T$, we have $R(x_n, x)\rightarrow 0$. We note that if $y\in T_i$, then there exists a sequence $(y_n)_{n\in\mathbb{N}}$ in $\psi_i({V}^*)$ with $\rho(y_n, y)\rightarrow 0$. Consequently
\begin{equation}\label{diamequal}
{\rm diam}_R T_i =\sup_{x,y \in T_i} R(x,y)=\sup_{x,y \in \psi_i({V}^*)} R^*(x,y)={\rm diam}_{R^*} \psi_i({V}^*).
\end{equation}
Applying this fact and Lemma \ref{diamas}, we have that, given $\varepsilon >0$, there exists an $n_0$ such that $\sup_{i\in\Sigma_{n_0}}{\rm diam}_R T_i< \varepsilon$. By Lemma \ref{neighbourhood}, we have that $x_n\in T_{n_0}(x)$, $\forall n\geq n_1$, for some $n_1$. It follows that $R(x_n, x)< \varepsilon$, for all $n \geq n_1$, and so $R(x_n,x)\rightarrow 0$ as desired.
\end{proof}

As noted above we have that the dendrite $T$ is the closure of $V^*$ with respect to $\rho$. Thus, under assumptions (A) and (B), the previous result implies that $(T,R)$ is the completion of $(V^*,R)$ and is a dendrite, $P$-a.s. Furthermore, because $T$ is a dendrite, it is possible to check from the definition of $R^*$ as a resistance metric that $R$ must be a shortest path metric (additive along paths), see \cite{Kigamidendrite}.

\section{The Hausdorff dimension of $T$}\label{dimension}

For a metric space, $(X,d)$, the {\it Hausdorff dimension} of $F\subseteq X$ is defined by
\begin{equation}\label{hausdorffdimension}
\mbox{dim}_H (F)  = \inf \left\{s \mbox{: }\mathcal{H}^s (F)<\infty\right\}= \sup \left\{s \mbox{: }\mathcal{H}^s (F)>0\right\},
\end{equation}
where $\mathcal{H}^s (F):= \lim _{\delta\rightarrow0} \inf \{ \sum_{i=1}^\infty \left( \mbox{diam}_d U_i \right)^s \mbox{: $\{U_i\}_i$ is a $\delta$-cover of $F$}\}$ and, for $\delta >0$, we call a finite or countable family of sets $(U_i)_{i=1}^{\infty}$ a {\it $\delta$-cover} of $F\subseteq X$ if $\mbox{diam}_d U_i \leq \delta$ for all $i$ and $F\subseteq \bigcup_i U_i$. If we suppose assumptions (A) and (B) both hold, then by Theorem \ref{rwelldef} we can construct the metric space $(T,R)$, $P$-a.s. In this section, we will present further conditions on the scaling factors $(w(i))_{i\in S}$ that allow us to calculate the Hausdorff dimension of $(T,R)$ to be constant and equal to $\alpha$, $P$-a.s., where $\alpha$, as in Lemma \ref{mthetal2}, is defined to be the unique positive solution to $F(\theta)=1$ for the function $F(\theta):=E(\sum_{i\in S}w(i)^\theta)$. Note that assumptions (A) and (B) imply that $\alpha \in (1, \infty)$. We start by demonstrating the upper bound, for which (A) and (B) are sufficient.

\begin{theorem} Suppose that (A) and (B) hold, then $\mbox{\rm{dim}}_H (T) \leq \alpha$, $P$-a.s.
\end{theorem}
\begin{proof} The representation of $\mbox{diam}_R T_i$ at (\ref{diamequal}) allows us to apply Lemma \ref{diamas} to deduce that, for large $n$, $(T_i)_{i\in\Sigma_n}$ is a $\delta$-cover for $T$. Thus
\[E\left( \mathcal{H}^\theta (T) \right) \leq  E\left( \liminf_{n\rightarrow \infty} \sum_{i\in\Sigma_n} \left( \mbox{diam}_R T_i \right)^\theta \right) \leq \liminf_{n\rightarrow \infty} E\left( \sum_{i\in\Sigma_n} \left(l(i) W_i  \right)^\theta \right)\]
where we have applied Fatou's lemma, (\ref{diamequal}) and the definition of $W_i$ for the second inequality. By Lemma \ref{wprops}, the expectation appearing in the right-hand side is equal to $F(\theta)^n E(W_\emptyset^\theta)$, and the second of these factors is finite. Furthermore, for $\theta>\alpha$, $F(\theta)<F(\alpha)=1$, and from this we can deduce that $\mathcal{H}^\theta (T)=0$, $P$-a.s. The result follows using the characterisation of the Hausdorff dimension at (\ref{hausdorffdimension}).
\end{proof}

To prove that the $\alpha$ defined above is also a lower bound for the Hausdorff dimension of $(T,R)$, we need to make more restrictive assumptions on the scaling factors, and we will derive the result in two special cases only. We proceed by applying a standard density result (see \cite{Falconerbook}, Proposition 4.9), and the first step in doing this involves constructing a suitable measure on $T$. The measure on $T$ that will be useful for our purposes will be a natural stochastically self-similar measure, and to construct it we proceed as in \cite{HamJon}, initially defining a measure on the shift-space $\Sigma$ and then applying the natural projection onto $T$. To characterise a measure on $\Sigma$, it is sufficient to define it on the cylinder sets, $i\Sigma:=\{ij:\:j\in\Sigma\}$, for $i\in\Sigma_*$. First, let $M^\theta_i$ be defined by the formula at (\ref{mitheta}) and also $M^\theta:=M^\theta_\emptyset$. Under assumptions (A) and (B) we can apply Lemma \ref{mthetal2} to obtain that, for $\theta\leq\alpha$, $P(M_\theta\in (0, \infty))=1$. Thus, if (A) and (B) hold, $\theta\leq \alpha$, and we define $\tilde{\mu}^\theta$ by
\[\tilde{\mu}^\theta(i\Sigma)=\frac{M^\theta_i l(i)^\theta}{M^\theta F(\theta)^n}\mbox{,\hspace{20pt}}i\in\Sigma_n,\]
then it is easy to apply the decomposition identity at (\ref{mthetaident}) to check that $\tilde{\mu}^\theta$ can be extended uniquely to a probability measure on $\Sigma$. Taking $\theta=\alpha$ removes the dependency on the index length, which suggests that this is the natural exponent to choose.

The projection, $\pi$, from $\Sigma$ to $T$ that we apply is obtained from the fact that for each $i\in\Sigma$, there exists a unique $\pi(i)\in T$ such that $\{\pi(i)\} = \bigcap_{n=1}^\infty T_{i|n}$, see \cite{Kigami}, Proposition 1.3.3. It is the measure $\mu^\alpha:=\tilde{\mu}^\alpha\circ\pi^{-1}$ that we will utilise in proving lower bound results for the Hausdorff dimension of $T$. Since the map $\pi$ is standard, we will not discuss its measurability (it is in fact a continuous surjection, \cite{Kigami}), however we will note that it is possible to check that $\mu^\alpha$ is a non-atomic Borel probability measure on $(T,R)$, which satisfies, for $i\in \Sigma_*$,
\begin{equation}\label{timeas}
\mu^\alpha(T_i)=\frac{M^\alpha_i l(i)^\alpha}{M^\alpha},
\end{equation}
(at least $P$-a.s. when assumptions (A) and (B) hold).

So far we have been able to use the fixed graphs $(V^n,E^n)$ to approximate $T$. The problem with these discretisations is that the lengths $l(i)R_i$ of edges within the graphs are, in general, widely varying as $n\rightarrow \infty$. To try to limit this effect, it will be useful to consider graph approximations to $T$ for which we have some uniform control over the edge lengths. This technique is also applied in \cite{Hambly}, Section 4, for example. Let us start by saying that $\Lambda\subseteq\Sigma_*$ is a cut-set if for every $i\in \Sigma$, there is a unique $j\in\Lambda$ with $i||j|=j$, and there exists an $n$ such that $|j|\leq n$ for all $j\in\Lambda$. This final condition is included to ensure that there is only a countable number of cut-sets. In particular, we will be interested in the random cut-sets $(\Sigma_\delta)_{\delta>0}$ defined by
\[\Sigma_\delta := \{i \mbox{: }l(i)\leq \delta < l(i|(|i|-1))\}.\]
Note that if (B) holds, then Lemma \ref{heightlemma}(b) guarantees that this is indeed a cut-set for all $\delta>0$, $P$-a.s. We introduce the corresponding graphs $(V^\delta,E^\delta)$, defined by $V^\delta:=\cup_{i\in \Sigma_\delta}\psi_i(V^0)$ and $E^\delta:=\{\psi_i(V^0):\:i\in\Sigma_\delta\}$. It is easy to check that, for each $\delta>0$, $(V^\delta,E^\delta)$ is a graph tree, $P$-a.s. Heuristically, to construct $(V^\delta,E^\delta)$ we start with $(V^0,E^0)$ and if there is an edge $\psi_i(V^0)$ in the graph with $l(i)$ greater than $\delta$, then replace it by the three edges $(\psi_{ij}(V^0))_{j\in S}$ and continue until there are no edges left to replace.

It will be important to be able to estimate the $\mu^\alpha$-measure of balls of the form $B_R(x,\delta):=\{y\in T:\:R(x,y)<\delta\}$. To do this, we will use collections of the sets $(T_i)_{i\in\Sigma_\delta}$ to cover the balls $B_R(x,\delta)$. In a slight change of notation from (\ref{tndef}), for $x\in T$, define $T_\delta (x):= \bigcup \{T_i\mbox{: }i\in\Sigma_\delta\mbox{, }x\in T_i\}$, and a larger neighbourhood of $x$ by
\[\tilde{T}_{\delta,\varepsilon} (x) := \bigcup \{T_i\mbox{: }i\in\Sigma_\delta\mbox{, }R(T_i, T_\delta (x))<\delta \varepsilon \}.\]
The number of sets making up this union is $N_{\delta,\varepsilon} (x) := \# \{i\in\Sigma_\delta\mbox{: }T_i\subseteq\tilde{T}_{\delta,\varepsilon}(x)\}$. It is clear that $B_R (x,\delta) \subseteq \tilde{T}_{\delta/\varepsilon, \varepsilon}(x)$, and recalling from (\ref{timeas}) that $\mu^\alpha (T_i) = M^\alpha_i/(M^\alpha l(i)^\alpha)$ for $i\in \Sigma_*$, we therefore obtain
\begin{equation}\label{muupper}
\mu^\alpha (B_R(x,\delta)) \leq (M^\alpha)^{-1} \varepsilon^{-\alpha} \delta^\alpha N_{\delta/\varepsilon,\varepsilon}(x) \sup_{i\in \Sigma_{\delta/\varepsilon}} M_i^\alpha.
\end{equation}
To complete the argument, we will analyse the behaviour of $\sup_{i\in \Sigma_{\delta}} M_i^\alpha$ and $N_{\delta,\varepsilon}(x)$. In bounding the first of these terms, we require control over the growth of the mean of $|\Sigma_\delta|$. The next lemma provides this using a related age-dependent branching process.

\begin{lemma} \label{tdeltalemma} Under assumption (A), $E|\Sigma_\delta| \leq c \delta ^{-\alpha}$ for some constant $c$.
\end{lemma}
\begin{proof} Consider the following branching process. Start at time 0 with one particle, labeled $\emptyset$. A particle, $i$, has three children at times $(\sigma_i-\ln w(ij))_{j\in S}$ where $\sigma_i = -\ln l(i)$ is the birth time of $i$. Label the child born to $i$ at $\sigma_i-\ln w(ij)\equiv -\ln l(ij)$ by $ij$, noting that children may not be labeled in birth order. It is not necessary to define the time of dying explicitly in this proof. The independence assumptions on the $(w(i))_{i \in \Sigma_*\backslash\{\emptyset\}}$ means that this setup describes a general branching process in the sense of Jagers, \cite{Jagers}, Chapter 6, and furthermore, if $Y_t$ is defined to be the random variable counting the births before time $t$ then it is easy to check that $|\Sigma_\delta | \leq 3 Y_{-\ln\delta}$. Noting that the Malthusian parameter for the branching process is precisely the $\alpha$ defined by $E(\sum_{i\in S}w(i)^\alpha)=1$, standard arguments then give that $EY_t\leq c_1e^{\alpha t}$, for some constant $c_1$. Combining these facts yields the lemma.
\end{proof}

We now proceed with bounding the growth of $\sup_{i\in\Sigma_\delta} M_i^\alpha$ as $\delta\rightarrow 0$. To allow us to apply Borel-Cantelli arguments to deduce $P$-a.s. such as this, it is useful to choose a particular subsequence of $\delta$s to investigate. Henceforth we consider $(\delta_n)_{n\geq 0}$ defined by $\delta_n:=e^{-n}$.

\begin{lemma} \label{msuplemma} If (A) holds and $\beta>0$, then
 $\lim_{n\rightarrow\infty} \delta_n^\beta \sup_{i\in\Sigma_{\delta_n}} M^\alpha_i=0$, $P$-a.s.
\end{lemma}
\begin{proof} If $P(\sum_{i\in S} w(i)=1)=1$, then $M_i^\alpha\equiv1$, $P$-a.s. for all $i$ and so the result is obvious. Assume now that $P(\sum_{i\in S} w(i)=1)<1$. For each $i\in \Sigma_*$, define a subset of $\Sigma_*$ by $\Sigma_i:=\{ik\mbox{: }k\in\Sigma_*\}\backslash\{i\}$ and related $\sigma$-algebras $\mathcal{F}_i:=\sigma(w(j)\mbox{: }j\in\Sigma_i)$, $\mathcal{G}_i:=\sigma(w(j)\mbox{: }j\in\Sigma_*\backslash\Sigma_i)$. By definition, we have $\mathcal{F}_i \bot \mathcal{G}_i$. It is straightforward to check that $M^\alpha_i$ is $\mathcal{F}_i$-measurable and $\{\Sigma_\delta=\Lambda\}\in\mathcal{G}_i$ for any cut-set $\Lambda$ containing $i$. Thus, for $i\in \Lambda$, with $\Lambda$ a cut-set we have, $P(M^\alpha_i>\eta\mbox{, }\Sigma_\delta=\Lambda)=P(\Sigma_\delta=\Lambda)P(M^\alpha>\eta)$ for $\eta\geq0$.
From this we deduce, using the countability of cut-sets,
\begin{eqnarray*}
P\left(\sup_{i\in\Sigma_\delta} M^\alpha_i >\eta\right) &=& \sum_{\Lambda:\:\Lambda \: \rm a \:cutset} P\left(\sup_{i\in\Sigma_\delta} M^\alpha_i >\eta\mbox{, }\Sigma_\delta=\Lambda\right)\\
&\leq & \sum_{\Lambda:\:\Lambda \: \rm a \:cutset} \sum_{i\in\Lambda} P\left(M^\alpha_i >\eta\mbox{, }\Sigma_\delta=\Lambda\right)\\
&=& \sum_{\Lambda:\:\Lambda \: \rm a \:cutset} \sum_{i\in\Lambda} P\left(M^\alpha >\eta\right)P\left(\Sigma_\delta=\Lambda\right)\\
&=& P\left(M^\alpha >\eta\right) \sum_{\Lambda:\:\Lambda \: \rm a \:cutset} |\Lambda| P\left(\Sigma_\delta=\Lambda\right)\\
&=& P\left(M^\alpha >\eta\right) E|\Sigma_\delta |.
\end{eqnarray*}
Since $M^\alpha$ is the limit of a multiplicative cascade, we can check the conditions of \cite{Liu}, Theorem 2.1 to deduce that its distribution satisfies $P(M_\alpha >\eta)\leq e^{-c_1\eta^\gamma}$ for some constants $c_1$ and $\gamma$. Applying this and Lemma \ref{tdeltalemma}, we obtain $P\left(\sup_{i\in\Sigma_\delta} M^\alpha_i > \eta\right) \leq c_2e^{-c_1\eta^\gamma} \delta^{-\alpha}$, and the result can subsequently be obtained by Borel-Cantelli.
\end{proof}

We are now in a position to be able to prove the lower Hausdorff dimension bound in our first special case. To prove the result we assume that $w(1)+w(2)=1$, $P$-a.s., which is a strengthening of assumption (A), and also implies that the resistance perturbations are precisely 1, thereby eliminating one random variable from our consideration. We also assume (B) to allow us to construct the metric space $(T,R)$, $P$-a.s. The assumption that is most specifically related to the problems which arise in the computation of a lower bound for the Hausdorff dimension, however, is a uniform lower bound for $(w(i))_{i\in S}$. Calculations of this kind become difficult if parts of the fractal become, in some sense, too small too quickly, and by bounding the scaling factors uniformly below, we are able to prevent this from occurring here.

\begin{theorem} \label{lowerboundbounded} Suppose $w(1)+w(2)=1$, $P$-a.s., assumption (B) holds, and there exists an $\varepsilon>0$ such that $P(w(i)>\varepsilon,\:i\in S)=1$, then $\mbox{{\rm dim}}_H(T)=\alpha$, $P$-a.s.
\end{theorem}
\begin{proof} From \cite{Kigami}, Proposition 1.3.5, we have that the intersection of distinct sets from $(T_i)_{i\in \Sigma_\delta}$ can only happen at points in $V^\delta$. Consequently, if $i,j\in \Sigma_\delta$, then the distance between the sets $T_i$ and $T_j$ is
\begin{equation}\label{rdgdistance}
R(T_i, T_j)=\min_{x\in V^0_i, y\in V^0_j}\min_{\substack{i_1,\dots,i_n\in \Sigma^\delta: x_0\in \psi_{i_1}(V^0),y\in\psi_{i_n}(V^0)\\ \psi_{i_m}(V^0)\cap\psi_{i_{m+1}}(V^0)\neq\emptyset}}\sum_{m=1}^nl(i_m)R_{i_m},
\end{equation}
where the second $\min$ in the right-hand expression is the graph distance in $V^\delta$ and $E^\delta$ between the vertices $x$ and $y$ when an edge of the form $\psi_i(V^0)$ is weighted by $l(i)R_i$, and that this is equal to $R(x,y)$ is a result of the fact that $R$ is additive along paths. Since we have that $\delta\varepsilon<l(i)$ for every $i \in \Sigma_\delta$, all the edge weights in $(V^\delta,E^\delta)$ are greater than $\delta\varepsilon$ (using the fact that $R^i\equiv 1$). Thus, if $R(T_i, T_j)<\delta\varepsilon$, then $R(T_i, T_j)=0$. This means that, for $x\in T$, every $T_i \subseteq \tilde{T}_{\delta, \varepsilon}(x)$ has a non-zero intersection with some $T_j \subseteq T_\delta(x)$ (where $i,j \in \Sigma_\delta$). A simple consideration of the geometry of $T$ allows us to deduce that we must therefore have $N_{\delta, \varepsilon}(x) \leq 9$, $\forall x \in T$, $P$-a.s. Thus, applying the bound of (\ref{muupper}), for $r\in[\varepsilon \delta_{n+1}, \varepsilon \delta_n)$ we obtain
\[\frac{\mu^\alpha (B_R(x,r))}{r^s} \leq 9 \varepsilon^{-s} e^s M_\alpha^{-1}\delta_n^{\alpha-s}\sup_{i \in \Sigma_{\delta_n}} M_\alpha^i.\]
If $s<\alpha$, then Lemma \ref{msuplemma} therefore gives $\limsup_{r\rightarrow 0} r^{-s}\mu^\alpha (B_R(x,r))$ $=0$ and so \cite{Falconerbook}, Proposition 4.9 implies the result.
\end{proof}

For the second special case in which we prove a Hausdorff dimension lower bound, we assume the following:
\\
\smallskip
{\bf (C) }{\it The random variables $(w(i))_{i\in S}$ are independent and their distributions satisfy the following tail condition. If $p\in(0,1)$, there exists a constant $\varepsilon>0$ such that
\[P(w(i) \leq x \varepsilon|\:w(i)\leq x) \leq p\mbox{,\hspace{20pt}}\forall x \in (0,1]\mbox{, }i\in S.\]}

Again, this is an assumption which stops the fractal getting too small too quickly. Rather than bounding them uniformly below, as is the previous result, we assume independence of the scaling factors and restrict the amount of build up of mass close to zero in the distributions of the scaling factors. This independence allows us to use a percolation-type argument, which enables us to avoid having to impose a uniform lower bound. If $w(i)$ has distribution function $\Phi$, then the inequality of (C) is equivalent to $\Phi(\varepsilon x)\leq p\Phi(x)$, for every $x\in(0,1]$. From this, it is easy to see that if $\Phi$ is approximately polynomial (i.e. there exist constants $c_{1},c_{2}$ such that $c_{1} x^n \leq \Phi(x) \leq c_{2} x^n$), then (C) holds. An example of when the build up of mass is too great for this to hold is the distribution function $\Phi(x)=(1-\ln x)^{-1}$.
The following lemma gives another alternative characterisation of assumption (C) that will prove useful in obtaining negative moments for the resistance perturbations $(R_i)_{i\in\Sigma_*}$.

\begin{lemma} \label{assumceq} Let $X$ be a $(0,1]$ valued random variable with distribution function $\Phi$, then the following statements are equivalent:\\
(a) If $p\in (0,1)$, then there exists an $\varepsilon\in(0,1)$ such that $\Phi(x\varepsilon)\leq p\Phi(x)$ for $x\in(0,1]$.
\\
(b) There exist constants $\varepsilon\in(0,1)$ and $\beta>0$ such that
\[E\left(\left(1-\frac{x^\beta\varepsilon}{X^\beta}\right)\mathbf{1}_{\{X\leq x\}}\right)\geq0,\hspace{20pt}\forall x\in(0,1].\]
\end{lemma}
\begin{proof} Assume (a) holds and fix $p\in (0,1)$. Choose $\varepsilon\in (0,1)$ such that $\Phi(x\varepsilon)\leq p\Phi(x)$ holds, and let $\beta>0$ satisfy $p<\varepsilon^\beta$. Integration by parts yields
\[E\left(x^\beta X^{-\beta}\mathbf{1}_{\{X\leq x\}}\right)=\lim_{\delta\rightarrow 0} \left\{ \left[ x^\beta y^{-\beta} \Phi(y)\right]_{y=\delta}^x+\beta\int_\delta^x x^\beta y^{-\beta-1} \Phi(y)dy\right\}.\]
Now, $\Phi(\varepsilon^n)\leq p^n$, and so, for $y\in(\varepsilon^{n+1}, \varepsilon^n]$, we have $y^{-\beta}\Phi(y)\leq \varepsilon^{-\beta(n+1)}p^n$. It follows that, because $p\varepsilon^{-\beta}<1$, $\lim_{\delta\rightarrow 0} \left[ x^\beta y^{-\beta} \Phi(y)\right]_{y=\delta}^x=\Phi(x)$, which gives an alternative expression for the first term in the limit above. Furthermore,
\begin{eqnarray}
\lim_{\delta\rightarrow 0} \beta\int_\delta^x x^\beta y^{-\beta-1} \Phi(y)dy&=&\sum_{n=0}^\infty \beta x^\beta\int_{x\varepsilon^{n+1}}^{x\varepsilon^n} y^{-\beta-1} \Phi(y)dy\nonumber\\
&\leq& \beta x^\beta \sum_{n=0}^\infty \int_{x\varepsilon^{n+1}}^{x\varepsilon^n} (x\varepsilon^{n+1})^{-\beta-1} p^n \Phi(x)dy\nonumber\\
&\leq & \beta \Phi(x)\varepsilon^{-\beta-1}\sum_{n=0}^\infty (p\varepsilon^{-\beta})^n\nonumber\\
&=&\frac{\beta\Phi(x)}{\varepsilon^{\beta+1}(1-p\varepsilon^{-\beta})}.\nonumber
\end{eqnarray}
Hence $E(x^\beta X^{-\beta}\mathbf{1}_{\{X\leq x\}})$ is bounded above by a constant multiple of $\Phi(x)$, and therefore (b) holds.

Conversely, suppose that (b) holds for some $\varepsilon\in(0,1)$ and $\beta>0$. Fix $p\in(0,1)$ and define $\varepsilon':=(p\varepsilon)^{1/\beta}$. For $x\in (0,1]$, we have
\[\Phi(\varepsilon'x)\leq E\left(\mathbf{1}_{\{X\leq \varepsilon'x\}}\frac{{\varepsilon'}^\beta x^\beta}{X^\beta}\right)\leq p E\left(\mathbf{1}_{\{X\leq x\}}\frac{\varepsilon x^\beta}{X^\beta}\right)\leq p \Phi(x),\]
which is statement (a).
\end{proof}

\begin{lemma} \label{rnegmoments} Under assumptions (A) and (C), $E((R_\emptyset)^{-d})<\infty$, for every $d>0$.
\end{lemma}
\begin{proof} The result will follow from Corollary \ref{rdnegmoments} if we can show that $E(w(i)^{-\beta})<\infty$, $i=1,2$, for some $\beta>0$. Under (C), this is a simple consequence of the previous lemma.
\end{proof}

We now use the alternative description of assumption (C) provided by Lemma \ref{assumceq} to show that the inequality of (C) still holds if the $w(i)$ are multiplied by the resistance perturbations $R_i$. We shall use the $\varepsilon_0$ obtained in the following lemma to describe what constitutes a small edge of $(V^\delta,E^\delta)$.

\begin{lemma} \label{assumcprop} If (A) and (C) hold and $p\in(0,1)$, then there exists an $\varepsilon_0>0$ such that
\[P(w(i)R_i \leq \varepsilon_0 x|\:w(i)\leq x) \leq p\mbox{,\hspace{20pt}}\forall x \in (0,1]\mbox{, }i\in S.\]
\end{lemma}
\begin{proof} By Lemma \ref{assumceq}, we can choose $\varepsilon, \beta>0$ such that
\[E\left(\left(1-\frac{\varepsilon x^\beta}{w(i)^\beta}\right)\mathbf{1}_{\{w(i)\leq x\}}\right)\geq0,\hspace{20pt}\forall x\in(0,1],\: i\in S.\]
Note also that, by Lemma \ref{rnegmoments}, $E((R_i)^{-\beta})<\infty$. Hence
\begin{eqnarray*}
P(w(i)R_i \leq\varepsilon_0x,\:w(i)\leq x)&\leq&E\left( \left( \frac{\varepsilon_0x}{w(i)R_i}\right)^\beta \mathbf{1}_{\{w(i)\leq x\}}\right)\\
&=&{\varepsilon_0}^\beta E(R_\emptyset^{-\beta})E\left( \frac{x^\beta}{w(i)^\beta} \mathbf{1}_{\{w(i)\leq x\}}\right)\\
&\leq& \frac{{\varepsilon_0}^\beta}{{\varepsilon}} E(R_\emptyset^{-\beta})P\left( w(i)\leq x \right)
\end{eqnarray*}
and so the result holds for $\varepsilon_0$ chosen suitably small.
\end{proof}

Henceforth we consider $p$ to be a deterministic constant and choose $\varepsilon_0$ so that the claim of the previous lemma holds. For reasons that will become clear in the proof of Lemma \ref{localclustertail}, it will be convenient to assume that $p\in(0,\frac{3^3}{4^4})$.

We now look to bound $N_{\delta, \varepsilon_0}(x)$, and to do so it will be convenient to use the language of percolation theory. We first define the events $(A_i^\delta)_{i\in\Sigma_\delta}$ by $A_i^\delta:=\{l(i)R_i\leq \varepsilon_0\delta\}$. For $i\in\Sigma_\delta$, we call the set $T_i$ {\it open} if $A_i^\delta$ occurs, and {\it closed} otherwise. Thus $T_i$ being open corresponds to $\psi_i(V^0)$ being a small edge in the graph $(V^\delta,E^\delta)$ (when an edge of this form is weighted by $l(i)R_i$). We will show that the largest cluster of sets from $(T_i)_{i\in\Sigma_\delta}$ which are open is not too large, and explain how this fact gives us a useful estimate for $N_{\delta, \varepsilon_0}(x)$.

Consider the random variable
\[H_\delta:= \left( \Sigma_\delta;\: (l(i|(|i|-1)))_{i\in\Sigma_\delta}\right).\]
We shall be conditioning on $H_\delta$, the informal motivation for doing so is the following. In the proof of Lemma \ref{tdeltalemma} we introduced a branching process where the individual $i$ is born at time $-\ln l(i)$. Hence if we stop the branching process at time $-\ln \delta$ (and can not see into the future) then we will be able to ascertain the value of $H_\delta$. However, we will not be able to observe the exact values of $l(i)$ or $R_i$ for $i\in\Sigma_\delta$. So, in this sense, we can consider $H_\delta$ to be the information about the weighted graph $(V^\delta,E^\delta)$ available at time $-\ln \delta$.

We now make precise the nature of the percolation-type behaviour that the independence of the $w(i)$s under the assumption (C) induces on the open/closed sets of $(T_i)_{i\in\Sigma_\delta}$. Note that the result provides an upper bound on the probability of a set from $(T_i)_{i\in\Sigma_\delta}$ being open which is independent of $\delta$. This scale-invariance property will be of particular importance for the arguments that follow.

\begin{lemma}\label{perclem} Suppose (A), (B) and (C) hold. Let $\delta\in(0,1)$. Conditionally on $H_\delta$, the sets $(T_i)_{i\in\Sigma_\delta}$ are open/closed independently. Furthermore, for $i\in\Sigma_\delta$ we have that ${P}\left(A^\delta_i\:|\:H_\delta\right)\leq p$, $P$-a.s., and, for $s\geq 1$, $
{E}\left(s^{\mathbf{1}_{A^\delta_i}}\:|\:H_\delta\right)\leq 1-p +sp$, $P$-a.s.
\end{lemma}
\begin{proof} Suppose that $i^1,\dots,i^n$ are distinct elements of $\Sigma_\delta$. Applying the independence of the $(w(i))_{i\in\Sigma_*\backslash\{\emptyset\}}$, elementary arguments yield
\[{{P}\left(A^\delta_{i^1},\dots,A^\delta_{i^n}\:|\:H_\delta\right)}=\prod_{m=1}^n{P}\left({w(i^m)R_{i^m}}\leq\frac{\varepsilon_0\delta}{x}\vline\: w(i^m)\leq \frac{\delta}{x}\right)_{x=l(i^m|(|i^m|-1))}.\]
This implies the independence claim. Consider the case $n=1$, and write $i=i^1$. Since $i\in\Sigma_\delta$, we must have $l(i|(|i|-1))>\delta$. Hence we can apply the bound of Lemma \ref{assumcprop} to the above expression to obtain that ${P}(A_i^\delta\:|\:H_\delta)\leq p$, ${P}$-a.s. The generating function bound is a simple consequence of this.
\end{proof}

We now introduce an algorithm to find the largest cluster of open sets of the form $(T_i)_{i\in\Sigma_\delta}$. We shall work on the graph $(\Sigma_\delta, \Gamma_\delta)$, where the edge set $\Gamma_\delta$ is defined by
\[\Gamma_\delta:=\{\{i,j\}:\:i,j\in\Sigma_\delta,\:T_i\cap T_j\neq \emptyset,\: T_i,T_j\mbox{ open}\}.\]
We shall write $\mathcal{C}(i)$ for the component of $(\Sigma_\delta,\Gamma_\delta)$ which contains the vertex $i$. Clearly, if $T_i$ is closed, then $\mathcal{C}(i)=\{i\}$. The following argument to find the size of the largest cluster is inspired by similar procedures used in \cite{Karp} to find the size of the largest cluster of a random digraph, and in \cite{AloSpe} to find the size of the largest cluster of a complete graph with edge percolation.

Assume that $H_\delta$ is known. Let $i\in\Sigma_\delta$ and set $L_0:=\{i\}$, $D_0:=\emptyset$. For $n\geq 1$, we define $L_n$, $D_n$ inductively. Assume we are given $L_n$, $D_n$. If $L_n\neq \emptyset$, then pick a vertex $j\in L_n$ (we can assume that there is a deterministic rule for doing this), and set
$$L_{n+1}:=L_n\cup\{k\in\Sigma_\delta:\:k\not\in L_n\cup D_n,\:\{j,k\}\in\Gamma_\delta\}\backslash\{j\},\hspace{10pt}D_{n+1}:=D_n\cup\{j\}.$$
If $L_n=\emptyset$, then set $L_{n+1}:=\emptyset$, $D_{n+1}:=D_n$.

It is a little unclear from this description as to exactly what the algorithm is doing and so we now try to provide a more intuitive description in terms of a branching process related to $\Sigma_\delta$. Call $i$ a {\it live}\index{live} vertex. For the first step, connect to $i$ all those vertices in $\Sigma_\delta$ that are joined to $i$ by an edge in $\Gamma_\delta$. Call these vertices live and $i$ {\it dead}\index{dead}. At an arbitrary stage, pick a live vertex, $j$, and connect to it all those vertices which we have not yet considered and are connected to $j$ by an edge in $\Gamma_\delta$. Call the new vertices in our branching process live and $j$ dead. Continue until we have no live vertices to pick from. At the point of termination, the collection of dead vertices contains exactly the vertices of $\mathcal{C}(i)$.

In our notation, $L_n$ represents the live vertices and $D_n$ the dead ones. Since we can pick each vertex in $\Sigma_\delta$ only once in the algorithm, we must have $D_{|\Sigma_\delta|+1}=\mathcal{C}(i)$. However, the algorithm may effectively terminate before this stage, giving that $|D_n|= n\wedge \tau$, where $\tau:=\inf\{n:\:L_n=\emptyset\}$. Necessarily $L_{|\Sigma_\delta|+1}=\emptyset$, and so this infimum is well-defined and finite. In particular, we must have $|\mathcal{C}(i)|=\tau$.

Using this algorithm, we are able to obtain a tail estimate for the distribution of $|\mathcal{C}(i)|$, conditional on $H_\delta$. Note that this result is scale-invariant; the tail bound on the size of a cluster does not depend on $\delta$.

\begin{lemma} \label{localclustertail} Suppose (A), (B) and (C) hold and let $\delta\in(0,1)$. There exists a deterministic constant $c$, not depending on $\delta$, such that, for $i\in\Sigma_\delta$,
\[{P}(|\mathcal{C}(i)|>n\:|\:H_\delta)\leq e^{-c n},\hspace{20pt}{P}\mbox{-a.s.}\]
\end{lemma}
\begin{proof} Choose $i\in\Sigma_\delta$ and use the algorithm described prior to this lemma to construct $(L_n, D_n)_{n\geq 0}$. Given $L_n, D_n$, the number of new live vertices in the $(n+1)$st step of the algorithm is
\[Z_n:=\#\{k\in\Sigma_\delta:\:k\not\in L_n\cup D_n,\:\{j,k\}\in\Gamma_\delta\},\]
if $L_n\neq \emptyset$, where $j=j(L_n)$ is the vertex chosen from $L_n$ in the algorithm, and 0 otherwise. On $\{L_n=\emptyset\}$, for $s\geq 1$, ${E}\left(s^{Z_n}\:\vline\:H_\delta, L_n, D_n\right)=1$, ${P}$-a.s. On $\{L_n\neq\emptyset\}$ with $j=j(L_n)$, using the independence and generating function bound of Lemma \ref{perclem}, for $s\geq 1$, we $P$-a.s. have
\begin{equation}
{E}\left(s^{Z_n}\:\vline\:H_\delta, L_n, D_n\right)\leq\prod_{\substack{k\in\Sigma_\delta:\:k\not\in L_n\cup D_n,\\ \:T_j\cap T_k\neq\emptyset}}{E}\left(s^{\mathbf{1}_{\{T_k\:\rm open}\}}\:\vline\:H_\delta\right)\leq(1-p+sp)^{4},\label{fridge}
\end{equation}
where for the second inequality we have applied the facts that at most three of the sets $(T_i)_{i\in\Sigma_\delta}$ intersect at any point, and also that $T_j$ only intersects with other elements of $(T_i)_{i\in\Sigma_\delta}$ at points in $\psi_j(V^0)$. Hence, because this upper bound is larger than 1, we have that ${E}\left(s^{Z_n}\:\vline\:H_\delta, L_n, D_n\right)\leq(1-p+sp)^{4}$, ${P}$-a.s.

For $n\leq\tau$ we have $|L_n| = |L_{n-1}|+Z_{n-1}-1$, and so, for $s\geq1$,
\begin{eqnarray*}
{E}\left(s^{|L_n|}\mathbf{1}_{\{|L_n|>0\}}\:\vline\:H_\delta\right)&\leq&{E}\left(s^{|L_n|}\mathbf{1}_{\{|L_{n-1}|>0\}}\:\vline\:H_\delta\right)\\
&=&{E}\left(s^{|L_{n-1}|}\mathbf{1}_{\{|L_{n-1}|>0\}}{E}\left( s^{Z_{n-1}-1}\:\vline\:H_\delta,L_{n-1}, D_{n-1} \right)\vline\:H_\delta\right)\\
&\leq& s^{-1}(1-p+sp)^{4}{E}\left(s^{|L_{n-1}|}\mathbf{1}_{\{|L_{n-1}|>0\}}\:\vline\:H_\delta\right),
\end{eqnarray*}
where we use the inequality at (\ref{fridge}) for the final bound and we have also used the fact that $\{|L_n|>0\}=\{\tau>n\}$. Applying this repeatedly yields a $P$-a.s. upper bound of $s^{-n}(1-p+sp)^{4n}$ for the expectation considered. Consequently, ${P}$-a.s., for $s\geq1$,
$${P}(|\mathcal{C}(i)|>n\:|\:H_\delta)={P}(|L_n|>0\:|\:H_\delta)\leq{E}\left(s^{|L_n|}\mathbf{1}_{\{|L_n|>0\}}\:\vline\:H_\delta\right)
\leq s^{-n}(1-p+ps)^{4n}.$$
This is minimised by $s=(1-p)/3p$, which is greater than 1, because of the upper bound we have assumed on $p$. Substituting for this value of $s$ we obtain an upper bound of $(4^43^{-3}(1-p)^3p)^n$ for ${P}(|\mathcal{C}(i)|>n\:|\:H_\delta)$, and the result follows.
\end{proof}

This lemma is easily extended to give a tail estimate for the distribution of the size of the {largest component}, $\mathcal{C}_{\delta}:=\sup_{i\in\Sigma_\delta}\mathcal{C}(i)$, from which we can prove the following almost-sure convergence result.

\begin{lemma}\label{clustertail} If (A), (B) and (C) hold, then $\limsup_{n\rightarrow\infty}n^{-1}\mathcal{C}_{\delta_n}<\infty$, $P$-a.s.
\end{lemma}
\begin{proof} Applying the conditional tail distribution of Lemma \ref{localclustertail}, we have
\[{P}(\mathcal{C}_\delta>n)={E}\left({P}(\mathcal{C}_\delta>n|H_\delta)\right)\leq {E}\left(\sum_{i\in \Sigma_\delta} {P}(|\mathcal{C}(i)|>n|H_\delta)\right)\leq {E}(|\Sigma_\delta|)e^{-c_{1}n},\]
and so it is possible to deduce from Lemma \ref{tdeltalemma} a tail bound of the form ${P}(\mathcal{C}_\delta>n)\leq c_{2}e^{-c_{1}n}\delta^{-\alpha}$ for $\mathcal{C}_\delta$. Applying this, a simple Borel-Cantelli argument yields the lemma.
\end{proof}

We are now able to prove the lower bound for the Hausdorff dimension of $T$ in the second special case.

\begin{theorem} \label{hausdorfflower2} Assume (A), (B) and (C), then ${\rm dim}_H(T)=\alpha$, ${P}${-a.s.}
\end{theorem}
\begin{proof} As at (\ref{rdgdistance}), the distance between sets of the form $(T_i)_{i\in\Sigma_\delta}$ is the weighted graph distance between the corresponding vertices in $({V}^\delta,{E}^\delta)$. Hence if it happens that $R(T_i, T_j)<\delta\varepsilon_0$, then the shortest path between a vertex of ${V}^0_i$ and a vertex of ${V}^{0}_j$ contains only edges contained in open sets from $(T_k)_{k\in\Sigma_\delta}$. Thus, if $T_k\subseteq\tilde{T}_{\delta,\varepsilon_0}(x)$ for $x\in T$, then there exists $i\in\Sigma_\delta$, $j\in\mathcal{C}(i)$ such that $T_k\cap T_j\neq \emptyset$ and $T_i\cap T_\delta(x)\neq\emptyset$. The number of intersections of sets of the form $(T_i)_{i\in\Sigma_\delta}$ was estimated in the proof of Lemma \ref{localclustertail}, and from this we can deduce that $N_{\delta, \varepsilon_0}(x)\leq 16\mathcal{C}_\delta$ uniformly in $x$. Consequently, for $r\in[\varepsilon_0 \delta_{n+1},\varepsilon_0 \delta_n)$, the bound at (\ref{muupper}) implies
\[\mu^\alpha (B_R(x,r)) \leq c(M^\alpha)^{-1}r^{\alpha} \mathcal{C}_{\delta_n}\sup_{i \in \Sigma_{\delta_n}} M^\alpha_i.\]
On recalling the conclusions of Lemmas \ref{tdeltalemma} and \ref{clustertail}, we are able to deduce from this bound that, for $s<\alpha$, $\limsup_{r\rightarrow0} r^{-s}\mu^\alpha (B_R(x,r))=0$, for every $x\in T$, ${P}$-a.s. The result is subsequently obtained by applying \cite{Falconerbook}, Proposition 4.9.
\end{proof}

\def\cprime{$'$}

\end{document}